\documentclass[12pt,a4paper]{amsart}
\usepackage{amsmath}
\usepackage{amsfonts}
\usepackage{amsthm}
\def\Z{{\mathbb Z}}
\def\C{{\mathbb C}}
\def\R{{\mathbb R}}
\def\acts{\triangleright}
\def\T{{\mathbb T}}

\def\CA{{\mathcal A}}
\def\CB{{\mathcal B}}

\def\CH{{\mathcal H}}

\def\beps{\boldsymbol{\epsilon}}

\DeclareMathSymbol\crossrt{\mathrel}{AMSb}{"6E}
\DeclareMathSymbol\crosslt{\mathrel}{AMSb}{"6F}

\def\lcross{\crosslt}
\def\id{{\mathrm i}{\mathrm d}}

\def\oh{\frac{1}{2}}

\def\acts{\triangleright}

\newtheorem{lemma}{Lemma}[section]
\newtheorem{theorem}[lemma]{Theorem}

\newtheorem{definition}[lemma]{Definition}

\newtheorem{remark}[lemma]{Remark}

\def\mm#1#2#3#4{\left( \begin{array}{cc} #1 & #2 \\ #3 & #4 \end{array} \right)}

\title[Spectral Triples over Bieberbach]{Real Spectral Triples over Noncommutative Bieberbach Manifolds}
\author[P.Olczykowski]{Piotr Olczykowski$^{\ast}$}
\email{piotr.olczykowski@uj.edu.pl}
\address{Copernicus Center for Interdisciplinary Studies, S\l{}awkowska 17, 31-016 Krak\'ow, Poland}
\author[A.Sitarz]{Andrzej Sitarz$^{\dagger,\ddagger}$}
\email{andrzej.sitarz@uj.edu.pl}
\address{Institute of Physics, Jagiellonian University,
Reymonta 4, 30-059 Krak\'ow, Poland}
\thanks{${}^\dagger$Most of this work was carried out at Institute of Mathematics of the
Polish Academy of Sciences, \'Sniadeckich 8, Warszawa, 00-950 Poland}
\thanks{${}^\ddagger$Partially supported NCN grant 2011/01/B/ST1/06474}
\thanks{$^{\ast}$Supported by the grant from The John Templeton Foundation.}
\begin{document}
\maketitle
\begin{abstract}
We classify and construct all real spectral triples over noncommutative Bieberbach 
manifolds, which are restrictions of irreducible real equivariant spectral triple over 
the noncommutative three-torus. We show, that in the classical case, the constructed
geometries correspond exactly to spin structures over Bieberbach manifolds and
the Dirac operators constructed for a flat metric.
\end{abstract}
\maketitle

\section{Introduction}

Bieberbach manifolds are compact manifolds, which are quotients of the Euclidean space by 
a free, properly discontinuous and isometric action of a discrete group. The torus is the canonical
example of a Bieberbach manifold, however, the first nontrivial example appears in dimension 
$2$ and is a Klein bottle. The case $d=3$ has already been described in the seminal works of 
Bieberbach \cite{Bieb1,Bieb2}. In this paper we work with the dual picture, looking at the suitable algebra
of functions on the Bieberbach manifold (and their noncommutative counterparts) in terms of fixed point 
subalgebra of the relevant dense subalgebra of the $C^\ast$ algebra of continuous functions on the 
three-torus and its corresponding noncommutative deformation ${\mathcal A}(\T^3_\theta)$. 

\subsection{Noncommutative Bieberbach Manifolds}

In this section we shall briefly recall the description of three-dimensional noncommutative Bieberbach manifolds as quotients of the 
three-dimensional noncommutative tori by the action of a finite discrete group. For details we refer to \cite{OlSi2}, here we present 
the notation and the results. Out of 10 different Bieberbach three-dimensional manifolds, (six orientable, including the three-torus 
itself and four nonorientable ones) only six have noncommutative counterparts. 

\begin{definition}[see \cite{OlSi2}]
Let $\CA(\T^3_\theta)$ be an algebra of smooth elements on a three-dimensional noncommutative torus, 
which contains the polynomial algebra generated by three unitaries $U,V,W$ satisfying relations,
$$ UV=VU, \quad UW=WU, \quad WV=e^{2\pi i\theta}VW, $$
where $0 < \theta < 1$ is irrational. We define the algebras of noncommutative Bieberbach manifolds 
as the fixed point algebras of the following actions of finite groups $G$ on $\CA(\T^3_\theta)$, 
which are combined in the table \ref{table1}. 
\begin{table}[here]
\centering
\begin{tabular}{|c|c|c|c|l|}
\hline
name & $\Z_N$ & $h^N \!=\! \id$ & action of $\Z_N$ on $U,V,W$  \\ \hline \hline
$\mathrm{B2}_\theta$& $\Z_2$ & h & $h \acts U = -U$, $h \acts V = V^*$, $h \acts W = W^*$, \\ \hline
$\mathrm{B3}_\theta$& $\Z_3$ & h & $h \acts U =  e^{\frac{2}{3} \pi i} U$, $h \acts V = e^{-\pi i \theta} V^* W$, 
                   $h \acts W = V^*$, \\ \hline
$\mathrm{B4}_\theta$& $\Z_4$ & h & $h \acts U =  i U$, $h \acts V = W$, $h \acts W = V^*$, \\ \hline
$\mathrm{B6}_\theta$& $\Z_6$ & h & $h \acts U =  e^{\frac{1}{3}\pi i} U$, $h \acts V = W$ , $h \acts W = e^{-\pi i \theta} V^*W$, \\ \hline
$\mathrm{N1}_\theta$& $\Z_2$ & h & $h \acts U = U^*$, $h \acts V = -V$, $h \acts W = W$, \\ \hline
$\mathrm{N2}_\theta$& $\Z_2$ & h & $h \acts U = U^*$, $h \acts V = -V $, $h \acts W = W U^*$, \\  \hline
\end{tabular}
\caption{The action of finite cyclic groups on $A(T^3_\theta)$}
\label{table1}
\end{table}
\end{definition}

We have shown in \cite{OlSi2} that the actions of the cyclic groups $\Z_N$, $N=2,3,4,6$ on the noncommutative  
three-torus, as given in the table \ref{table1} is free. The aim of this paper is 
to study and classify flat (i.e. restricted from flat geometries of the torus $\CA(\T^3_\theta)$)
real spectral geometries over the orientable noncommutative Bieberbachs.
\section{Spectral triples over Bieberbachs}

Since each noncommutative Bieberbach algebra is a subalgebra of the noncommutative torus, 
a restriction of the spectral triple over the latter to the subalgebra, gives a generic spectral triple 
over a noncomutative Bieberbach manifold, which might be, however, reducible. By restriction
of a spectral triple $(\CA,\pi,\CH,D,J)$ to a subalgebra, $\CB \subset \CA$, we understand the 
triple $(\CB,\pi,\CH',D',J')$ where $\pi'$ is the restriction of $\pi$ to $\CB$, $\CH' \subset \CH$ is 
a subspace invariant under the action of $\CB$,$D$ and $J$, so that $D',J'$ are their restrictions
to $\CH'$ (note that in the even case this must apple also to $\gamma$). 

In what follows we shall show that, in fact, each spectral triple over Bieberbach is a restriction of a spectral triple over the torus, first showing that each spectral triple over Bieberbach can be lifted 
to a noncommutative torus.

\subsection{The lift and the restriction of spectral triples}

\begin{lemma} 
Let $(\mathrm{B}N_\Theta,\CH,J,D )$ be a real spectral triple over a noncommutative 
Bieberbach manifold $\mathrm{B}N_\theta$. Then, there exists a spectral triple over three-torus, 
such that this triple is its restriction.
\end{lemma}

\begin{proof}
In \cite{OlSi2} we showed that the crossed product algebra $\mathcal{A}(\T^3_\theta)\lcross\Z_N$
is isomorphic to the matrix algebra of the noncommutative Bieberbach manifold algebra:
$$ \CA(\T^3_\theta) \lcross \Z_N \sim \mathrm{B}N_\theta \otimes M_N(\C). $$ 

First, let us recall that any spectral triple $\CA,\CH,D,J$ could be lifted to 
a spectral triple over $\CA \otimes M_n(\C)$. Indeed, if we take $\CH' = \CH \otimes M_n(\C)$
with the natural representation $\pi'(a \otimes m) (h \otimes M) = \pi(a)h \otimes mM$,
the diagonal Dirac operator and $J' (h \otimes M) = Jh \otimes U M^\dagger U^\dagger$, for
an arbitrary unitary $U \in M_n(\C)$ it is easy to see that we obtain again a real spectral
triple. Applying this to the case of ${\mathrm B}N_\theta$, and identifying 
$\mathrm{B} N_\theta \otimes M_N(\C)$ using the above isomorphism we obtain 
a spectral triple over $\mathcal{A}(\T^3_\theta) \lcross \Z_N$. As $\CA(\T^3_\theta)$ is 
a subalgebra of $\mathcal{A}(\T^3_\theta) \lcross \Z_N$ by restriction we obtain, in turn, a spectral 
triple over a three-dimensional noncommutative
 torus. In fact, it is easy to see that 
we obtain a spectral triple, which is equivariant with respect to the action of $\Z_N$ group. 
Clearly, the fact that we have a representation of the crossed product algebra is 
just a rephrasing of the fact that we have a $\Z_N$-equivariant representation
 of $\mathcal{A}(\T^3_\theta)$.  By
 definition, the Dirac operator lifted from the spectral triple 
over ${\mathrm BN}_\theta$ commutes with the group elements (which are identified 
as matrices in $M_N(\C)$. A little care is required to show that the lift of $J$ would 
properly commute with the generator of $\Z_N$. However, since the lift of $J$ involves
a matrix $U$,  it is sufficient to use a matrix, which provides a unitary equivalence 
between the generator $h$ and its inverse $h^{-1}$ of $\Z_N$ in $M_N(\C)$. Simple
computation shows that the following $U$, $U_{00}=1,  U_{kl} = \delta_{k,N-l}$ for 
$k,l =1,\dots, N-1$ is the one providing the equivalence.

Hence, by this construction, we obtain a real, $\Z_N$-equivariant spectral triple over 
${\mathcal A}(\T^3_\theta)$. It is easy to see that the original spectral triple over ${\mathrm BN}_\theta$ is
a restriction of the constructed spectral triple by taking the invariant subalgebra
of ${\mathcal A}(\T^3_\theta)$, the $\Z_N$-invariant subspace of $\CH$ and the restriction of 
$D$ and $J$. 
\end{proof}

\begin{remark}
The procedure described above does not necessarily provides the {\em canonical}
(equivariant) Dirac operator over $\mathcal{A}(\T^3_\theta)$. Indeed, even a simple 
example of $\mathcal{A}(\T^1)$ shows that the lifted Dirac operator differs from the 
fully equivariant one by a bounded term. There may exist, however, a fully equivariant 
triple, so that its restriction is the same triple we started with. 
\end{remark}

\begin{definition}
We call the geometry (spectral triple) over the noncommutative Bieberbach manifold 
$\mathrm{B}N_\theta$ {\bf flat} if it is a restriction of a {\bf flat} spectral triple over the 
noncommutative three-torus, that is, the latter being equivariant with respect to the 
full action of $U(1) \times U(1) \times U(1)$.
\end{definition}

\subsection{Equivariant real spectral triples over  $C(T^3_\theta)$.}

Let us take one of the eight canonical equivariant spectral triples over the 
noncommutative torus $\mathcal{A}(\T^3_\theta)$ (for a definition of equivariance 
see \cite{Si00}, for classification of equivariant real spectral triples over 
a noncommutative two-torus see \cite{PaSi}, for a generalization to higher dimensions 
see \cite{Vens}). 

Let us recall, that the Hilbert space is spanned by two copies of $l^2(\Z^3)$, each 
of them with basis $e_\mu$, where $\mu$ 
is a three-index and each $\mu_1,\mu_2,\mu_3$ is either integer 
or half-integer depending on the choice of the spin structure. We parametrize 
spin structures by $\epsilon_i$ $i=1,2,3$, which can take values $0$ or $\oh$,
so that $\mu_i-\epsilon_i$ is always integer. \cite{Vens}. 

For  ${\bf k} = [k_1,k_2,k_3] \in\Z^3$ let us define the generic homogeneous element of 
the algebra of polynomials over the noncommutative torus:
$$ x^{\bf k} = e^{\pi i \theta k_2 k_3 }U^{k_1} V^{k_2} W^{k_3}.$$ 
We fix the representation of the algebra of the noncommutative torus 
(relevant for the construction  of noncommutative Bieberbach manifolds) 
on $l^2(\Z^3)$ to be:
as follows:
\begin{equation}
\pi (x^{\bf k}) e_\mu=e^{\pi i \theta (k_3 \mu_2  - \mu_3 k_2)} e_{\mu+k},
\label{rep}
\end{equation}    
and on the Hilbert space $\CH$ we take it diagonal:
$$ \pi(x) = \left( \begin{array}{cc}
\pi(x) & 0 \\ 0 & \pi(x) \end{array} \right).
$$

The real structure $J$ is of the form:
$$ J = \left( \begin{array}{cc}
0 & - J_0 \\
J_0 & 0 \end{array} \right), 
$$
where 
$$J_0 e_\mu = e_{-\mu}, \;\;\; \forall \mu \in \Z^3 + \epsilon. $$

The most general equivariant and real Dirac operator (up to rescaling) is of the form: 
\begin{equation}
 D = \left( \begin{array}{cc}
       R \delta_1 & \delta_2 + \tau \delta_3 \\
       \delta_2 + \tau^* \delta_3 & -R \delta_1 
\end{array} \right), 
\label{Dirac3}
\end{equation}
where $R$ is a real parameter and $\tau$ a complex parameter (parametrizing the
conformal structure of the underlying noncommutative $2$-torus). 

The derivations $\delta_i$, $i=1,2,3$ act diagonally on the $l^2(\Z^3)$:
$$ \delta_i e_\mu=\mu_i e_\mu . $$

\begin{theorem}[see \cite{Si00,PaSi,Vens}]
\label{T3triple}
The spectral triple, given by $(\CA_\theta,\pi,J,D, \CH)$ is an $U(1)^3$-equivariant,
irreducible, real spectral triple.
\end{theorem}

Note that the choice of $J$ and the Dirac operator (\ref{Dirac3}) has still
some unnecessary freedom. Indeed, changing $J$ to $-J$ does not influence
any of the axioms. Combining that with a conjugation by Pauli matrices $\sigma^2$
or $\sigma^3$ we see that we might restrict ourselves to the case $R>0$.

We shall discuss later the difference between the choice of $\tau$ or $\tau^*$.

\subsection{The equivariant action of $\Z_N$.}

Our aim in this part will be to construct all equivariant representations of the finite
groups $G=\Z_N$, $N=2,3,4,6$ on the Hilbert space of the spectral triple for
the noncommutative torus $\mathcal{A}(\T^3_\theta)$, which implement the group action
on the algebra. The conditions for the representation $\rho$ of $G$ are:
\begin{itemize}
\item action:
$$ \rho(g) \pi(a) = \pi(g \acts a) \rho(g), \;\;\; \forall g \in G, a \in \CA(\T^3_\theta),$$
\item $D$-equivariance:
$$ D \rho(g) = \rho(g) D, \;\;\; \forall g \in G, $$
\item $J$-equivariance:
$$ J \rho(g) = \rho(g) J, \;\;\; \forall g \in G.$$
\end{itemize}

We begin with two auxiliary lemmas.
\begin{lemma}\label{repdiag}
Any equivariant representation of $\Z_N$ on the Hilbert space of the spectral
triple must be diagonal:
$$
\rho(g) = \left( 
\begin{array}{cc}
\rho_+(g) & 0 \\ 0 & \rho_-(g)
\end{array} \right).
$$
and commute with $\delta_1 \otimes 1$.
\end{lemma}
\begin{proof}
First of all, observe that from the definition of equivariance the element 
$\sigma^1 = \frac{1}{R}  U^* [D,U]$ commutes with $\rho(g)$ for any $g \in G$:
$$ \rho(g) \left( U^* [D,U] \right) = \alpha_g(U^*) [D, \alpha(U)] \rho(g) = U^* [D,U] \rho(g).$$
This happens since the action of $g$ on $U$ is by multiplication by a root of $1$. Therefore 
the action of $\Z_N$ is diagonal. and we can treat the two copies independently. 

Moreover, $\sigma^1 D + D \sigma^1$ also commutes with $\rho$, and, since the latter
is proportional to $\delta_1 \otimes 1$, we obtain the second statement.
\end{proof}

Further, we have:

\begin{lemma}\label{replem}
Let $\rho_1, \rho_2$ be representation of a finite group $G$ on the Hilbert 
space $\CH$ such that both implement its action on $\CA(\T^3_\theta)$,
are real-equivariant and $D$-equivariant. Then $\rho_1(g) = T(g) \rho_2(g)$, 
for every $g \in G$, where $T(g) = \zeta \cdot 1 \in M_2(\C)$.
\end{lemma}

\begin{proof}
{}From the fact that both representations implement the action we see that
$T(g) = \rho_1(g) \rho_2(g^{-1})$ commutes with the representation of the 
algebra $\CA(\T^3_\theta)$. Since it $D, T(g)=0$, it commutes with
$V^* [D,V]$ and $W^*[D,W]$ and therefore with any matrix in $M_2(\C)$. 
Using similar arguments as in the Lemma \ref{repdiag} we see that it must
therefore commute with $\delta_2$ and $\delta_3$. Hence, $T(g)$ is an 
$u(1)^3$ invariant element, which commutes with $D$, $J$ and the representation 
of the algebra. On the other hand, we know that the spectral triple from Theorem \ref{T3triple} 
is equivariant-irreducible, so $T(g)$ must be proportional to the identity operator.
\end{proof}
The strategy of finding all equivariant representations will be based on finding one,
which is real-equivariant and then, with the of the above lemma, finding all possible
representations, which are real-equivariant and commute with $D$ at the same time.

\subsection{Spin structures and equivariant actions of $\Z_N$}

In this section we shall determine, which of the spin structures over the 
noncommutative torus may admit an equivariant action of $\Z_N$ for the
chosen values of $N$. This will be a first step only, as later we shall proceed
with further restrictions. Since the arguments and computations are basically 
identical in each case, we combine them into a single section.

We denote by $h$ the generator of the group $\Z_N$, $h^N=\id$. As it follows from 
Lemma (\ref{repdiag}) the representation $\rho$ of $\Z_N$ is diagonal: 
$$ \rho(h)   e_{\mu_1,\mu_2,\mu_3, \pm} = \rho_\pm(h)  e_{\mu_1,\mu_2,\mu_3},$$
with $\rho_\pm$ being operators on $l^2(\Z^3)$.

Additionally, due to  $J$-equivariance, they satisfy:
\begin{equation}
J_0 \rho_+(h) = \rho_-(h) J_0. 
\label{rhoJ}
\end{equation}

We shall introduce convenient notation for the action of the group $\Z_N$ on 
the generators (and the basis) of the algebra $\CA(\T^3_\theta)$, which uses
the shorthand presentation (\ref{rep}):
\begin{equation}
h \acts x^{\bf k} = e^{2 \pi i\frac{k_1}{N}} x^{[k_1, A [k_2,k_3]]},  
\label{repsh}
\end{equation}
where $k \in \Z^3$ and $A \in M_2(\Z)$ is the following matrix (for each
of $N=2,3,4,6$, respectively:
$$
\begin{array}{c|c|c|c}
2& 3 &4 &6 \\[2mm] \hline &&& \\[-2mm]
\mm{-1}{0}{0}{-1}&
\mm{-1}{-1}{1}{0}&
\mm{0}{-1}{1}{0}& 
\mm{0}{-1}{1}{1}
\end{array}.
$$
\begin{lemma}
\label{zetafix}
The only $D$-equivariant and real-equivariant action of $\Z_N$ on $\CH$,
which implements the action of $\Z_N$ on the algebra (Table \ref{table1}) is 
possible if:
\begin{center}
\begin{tabular}{|c||c|c|c|}
\hline
$N$       & $3$ & $4$ & $6$ \\ \hline \hline
$\tau$    & $e^{\pm\frac{1}{3}\pi i}$ & $e^{ \pm\frac{1}{2}\pi
 i}$ & $e^{\pm\frac{1}{3}\pi i}$\\ \hline
$\epsilon_2$ & $0$ & $\epsilon$ & $0$ \\ \hline 
$\epsilon_2$ & $0$ & $\epsilon$ & $0$ \\ \hline
\end{tabular}
\end{center}
with no restrictions on $\tau$ in the case $N=2$. 

For every such action, there exists $\zeta \in \C$ such that 
$\rho_-(h)=\zeta \rho_+(h)$:
\begin{center}
\begin{tabular}{|c|c|c|c|c|}
\hline
N & 2 & 3 & 4 & 6 \\ \hline
$\tau$ & -- & $e^{\pm \frac{1}{3}\pi i}$ & $e^{\pm \frac{1}{2}\pi i}$ & $e^{\pm
 \frac{1}{3}\pi i}$ \\ \hline
$\zeta$ & $-1$ & $-\tau$ & $\tau^*$ & $\tau^*$ \\ \hline
\end{tabular}
\end{center}
\end{lemma}
\begin{proof}
Consider $\omega_V = \pi(V^*) [D,\pi(V)])$ and $\omega_W = \pi(W^*) [D, \pi(W)])$:
$$ \omega_V = \mm{0}{1}{1}{0}, \;\;\;\; \omega_W = \mm{0}{\tau}{\tau^*}{0}. $$
Using the fact that $\rho$ implements the action of $\Z_N$ on the algebra we have:
\begin{equation}
 \left[ \begin{array}{cc} \rho(h) & 0 \\ 0&\rho(h)  \end{array} \right]
 \left[ \begin{array}{cc} \omega_V \\ \omega_W \end{array} \right] 
= A^T \left[ \begin{array}{cc} \omega_V \\ \omega_W \end{array} \right] 
 \left[ \begin{array}{cc} \rho(h) & 0 \\ 0 & \rho(h)  \end{array} \right]. 
\end{equation}
Rewriting the above identities we get, in each of the cases:
\begin{equation}
\begin{aligned}
N=2:& \;\;\;\;\; &\rho_+(h)  = - \rho_-(h),  \;\;\;\;\; \tau \rho_+ = - \rho_-(h) \tau, \\
N=3:& \;\;\;\;\; &\rho_+(h)  = (\tau-1) \rho_-(h),  \;\;\;\;\; \tau \rho_+ = - \rho_-(h), \\
N=4:& \;\;\;\;\; &\rho_+(h)  = \tau \rho_-(h),  \;\;\;\;\; \tau \rho_+ = - \rho_-(h), \\
N=6:& \;\;\;\;\; &\rho_+(h)  = \tau \rho_-(h),  \;\;\;\;\; \tau \rho_+ = (\tau-1) \rho_-(h).
\label{rhotau}
\end{aligned}
\end{equation}
which are self-consistent only if 
\begin{equation}
\begin{aligned}
N=3,6:& \;\;\;\;\; &\tau = e^{\pm \frac{1}{3}\pi i}, \\
N=4:& \;\;\;\;\; & \tau =   e^{\pm \frac{1}{2}\pi i }.
\label{taucon} 
\end{aligned}
\end{equation}
There are no restrictions for $\tau$ in the case $N=2$. 
\end{proof}

Next, consider the vector $e_{\beps} = e_{\epsilon_1,\epsilon_2,\epsilon_3}$, which is 
generating and separating vector for $l^2(\Z^3)$ for a given spin structure determined 
by $\beps$. In order to determine $\rho_\pm(h)$ it is sufficient to see its action on
$e_{\beps}$.

\begin{lemma}
There exist elements in $\CA(\T^3_\theta)$ (which are polynomials in the generators)
such that:
$$ \rho_+(h) e_{\beps} = \sigma_+ e_{\beps}, \;\;\; 
     \rho_-(h) e_{\beps} = \sigma_- e_{\beps}.$$
\end{lemma}
\begin{proof}
Clearly, since $\rho_\pm(h)$ are proportional to each other it is sufficient to show it for
$\rho_+(h)$. As the action of $\Z_N$ commutes with $D$, it commutes with $D^2$, which
is diagonal, hence $\rho_+(h)$ commutes with the restriction of $D^2$ to $l^2(\Z^3)$. The
vector $e_{\beps}$ is an eigenvector of $D^2$ and the eigenspace of the same eigenvalue 
is finite dimensional and spanned by the vectors $e_{\mu_1,\mu_2,\mu_3}$. Since each
of such vectors could be obtained by acting with a homogenous polynomial in the 
generators of the noncommutative three-torus on $e_{\beps}$ any vector in this subspace
is of the form $\sigma_+ e_{\beps}$, where $\sigma_+ \in \CA(\T^3_\theta)$ is a 
polynomial in the generators.
\end{proof}

\begin{theorem}
The action of $\Z_N$ on the Hilbert space of the equivariant spectral 
triple over $\CA(\T^3_\theta)$, which implements the action (\ref{repsh}) 
on the algebra and which commutes with $J$ and $D$ is possible only for 
some of the spin structures for the torus, parametrized by
$\epsilon_k =0, \oh$, $k=1,2,3$,
\begin{itemize}
\item For $N=3,6$: $\epsilon_2=\epsilon_3=0$.
\item For $N=4$: $\epsilon_2=\epsilon_3$.
\end{itemize}
\end{theorem}
\begin{remark}
Before we begin with the proof, observe that the above theorem list {\bf some} of
the necessary restrictions but not necessarily {\bf all} of them. We shall see that in
some cases ($N=2$, in particular) there are some additional constraints.
\end{remark}
\begin{proof}
We know that the representation of the discrete group $\Z_N$ on the generating vectors
$e_{\beps}$ is by the action of algebra elements $\sigma_\pm$, proportional to each 
other. We shall consider now the action of the commutator $[D, \rho(h)]$ on 
$e_{\beps, \pm}$. Since we have already established that $\rho(h)$, commutes with 
$\delta_1$, we have:
\begin{align*}
&\left( \zeta (\delta_2 + \tau \delta_3) \sigma_+ - \sigma_+ (\delta_2 +\tau \delta_3) \right) e_{\beps} = 0, \\
&\left( (\delta_2 + \tau^* \delta_3) \sigma_+ - \zeta \sigma_+ (\delta_2 +\tau^* \delta_3) \right) e_{\beps} = 0, 
\end{align*}
where we have already taken into account that $\sigma_- = \zeta \sigma_+$. If we call
$\partial = \delta_2 + \tau \delta_2$, these equations could be rewritten as:
\begin{align*}
&\left( \zeta (\partial(\sigma_+) + (\zeta -1) \sigma_+ \partial \right) e_{\beps} = 0, \\
&\left( (\partial^*(\sigma_+) + (1 - \zeta) \sigma_+ (\partial^*) \right) e_{\beps} = 0,
\end{align*}
and further, taking into account that $\partial e_{\beps} = (\epsilon_2 + \tau \epsilon_3)e_\epsilon$:
\begin{align*}
&\left( \zeta \partial(\sigma_+) + (\zeta -1) \sigma_+ (\epsilon_2 + \tau \epsilon_3) \right) e_{\beps} = 0, \\
&\left( \partial^*(\sigma_+) + (1 - \zeta) \sigma_+  (\epsilon_2 + \tau^* \epsilon_3) \right) e_{\beps} = 0.
\end{align*}
The elements acting on $e_{\beps}$ are polynomials from $\CA(\T^3_\theta)$, so, since 
$e_{\beps}$ is a separating vector they must vanish identically, therefore we obtain the
following equations:
\begin{align*}
& \zeta \partial(\sigma_+) = (1-\zeta )(\epsilon_2 + \tau \epsilon_3)  \sigma_+ , \\
& \partial^*(\sigma_+) = (\zeta - 1)  (\epsilon_2 + \tau^* \epsilon_3) \sigma_+,
\end{align*}
and we need to look for possible solutions in the polynomial algebra generated by
$U,V,W$ depending on the values of $\epsilon_2, \epsilon_3$ and $\zeta$, with the 
latter fixed in the Lemma (\ref{zetafix}).

First of all, observe that in the case $N=2$, $1-\zeta=2$ and there are no restrictions
on $\epsilon_2$ and $\epsilon_3$ and $\sigma_+$ proportional to $V^{-2 \epsilon_2} W^{-2 \epsilon_3}$
is the solution.

In the case of $N=3,4,6$ we find that the explicit equations are:
\begin{align*}
&N=3: \;\;\;\; 
\partial \sigma_+ = -  (\epsilon_2 + \epsilon_3 + \tau \epsilon_2 + \tau^* \epsilon_3) \sigma_+, \\
&N=4: \;\;\;\;
\partial \sigma_+ = -  (\epsilon_2 + \epsilon_3 + \tau \epsilon_2 + \tau^* \epsilon_3) \sigma_+, \\
&N=6: \;\;\;\;
\partial \sigma_+ =  (-\epsilon_2 + \tau (\epsilon_2 - \epsilon_3) + \tau^2 \epsilon_3) \sigma_+, \\
\end{align*}
and it is an easy exercise to check that they have nonzero solutions only in the
following cases:
\begin{equation}
\begin{aligned}
&N=3: \;\;\;\;\; &  
\epsilon_2 = \epsilon_3 = 0, \\
&N=4: \;\;\;\;\; &  
\epsilon_2 = \epsilon_3 = 0 \;\; \hbox{\it or\ \ } 
\epsilon_2 = \epsilon_3 = \oh, \\
&N=6: \;\;\;\;\; &  
\epsilon_2 = \epsilon_3 = 0.
 \end{aligned}
\end{equation}
\end{proof}

Let us summarize the results of this section: we have established that the desired
actions of $\Z_N$ exist only for some of the spin structures over the noncommutative
three-torus and, if they exist, they are unique up to multiplication by a scalar. This
will allow us to classify, case by case, all possible actions and hence, all possible
restrictions of the spectral triple over noncommutative torus to the fixed point subalgebra
of the noncommutative Bieberbachs.

\subsection{All equivariant actions of $\Z_N$}

Here we restrict ourselves to the actions only for the spin structures determined 
in the previous section and we proceed case by case. However, we begin a general
lemma:

\begin{lemma}
For a given spin structure on $\CA(\T^3_\theta)$ the following is a most general
real-equivariant and diagonal action of $\Z_N$ on $l^2(\Z^3)$, which implements
the action on the algebra from Table \ref{table1}:
\begin{equation}
\rho_\pm(h)  e_{\mu_1,\mu_2,\mu_3, \pm} 
   = \beta_{\pm} e^{\frac{2\pi i}{N}(\mu_1 \pm \epsilon_1)} e_{\mu_1, A[ \mu_2 , \mu_3],\pm}. 
   \label{fact}
\end{equation}
where $A$ is the matrix defined in section 2.4, and $\beta_\pm$ are:
$$ (\beta_\pm)^N =1, \;\;\;\; \beta_+ = (\beta_-)^*.$$
\end{lemma}
\begin{proof}
First we demonstrate that it implements the action, using (\ref{rep}) and (\ref{repsh}),
using the fact that the bilinear functional $\phi({\bf x}, {\bf y}) = x_3 y_2 - y_3 x_2$ is
invariant under $A$, that is $\phi({\bf x}, {\bf y}) = \phi(A{\bf x}, A{\bf y})$, where 
$x,y \in \mathbb{R}$. We have:
$$
\begin{aligned}[l]
\left( \rho_\pm(h) \right. & \left.\pi(x^{\bf k}) -  \pi(h \acts x^{\bf k})\rho_\pm(h) \right)  e_{\mu_1,\mu_2,\mu_3,\pm} = \\
 =&\; \beta_\pm \left( e^{\frac{2\pi i}{N}(\mu_1 \pm \epsilon_1 + k_1)}  -  
e^{2\pi i \frac{k_1}{N}} e^{\frac{2 \pi i}{N}(\mu_1 \pm \epsilon_1)}  \right)
e^{\pi i \theta(k_3 \mu_2 - \mu_3 k_2)} e_{\mu_1+k_1, A [\mu_2 + k_2, \mu_3+k_3],\pm} \\
= & \; 0.
\end{aligned}
$$
Furthermore, for the real-equivariance we have:
$$
\begin{aligned}[c]
\rho(h) J e_{\mu_1,\mu_2,\mu_3, \pm} 
&= \pm\beta_{\mp} e^{\frac{-2\pi i}{N}(\mu_1 \pm \epsilon_1)} e_{-\mu_1, -A[ \mu_2 , \mu_3], \mp}, \\
J \rho(h) e_{\mu_1,\mu_2,\mu_3, \pm} 
&= \pm\left( \beta_\pm  e^{\frac{2\pi i}{N}(\mu_1 \pm \epsilon_1)} \right)^*
e_{-\mu_1,A[-\mu_2, -\mu_3],\mp}.
\end{aligned}
$$
and that enforces the relation: $\beta_+ = (\beta_-)^*$.
\end{proof}
Next, let us check the invariance of $D$ with respect
 to the above action of $\Z_N$. 

\begin{lemma}
In addition to the constraints (\ref{zetafix}), the action (\ref{fact}) commutes with the Dirac operator
only if the following conditions are satisfied:
\begin{itemize}
\item $N=2$: $\epsilon_1=\oh$, $\beta_+ = \sigma$,
\item $N=3$:  $\tau = - e^{ \frac{2}{3} \pi i \sigma}, \;\;\;\;
\beta_+ = - e^{ \frac{4}{3} \pi i \epsilon_1} \tau = 
e^{ \frac{2}{3}
 \pi i (2 \epsilon_1 + \sigma)}, $ 
\item $N=4$: $\epsilon_1=\oh$ and
$ \tau = e^{ \frac{\sigma}{2} \pi i}, \;\;
    \beta_+ = \kappa e^{\frac{1}{4} \pi i (\sigma -1)},$ 
\item $N=6$:
$\epsilon_1=\oh$ and $ \tau =  e^{\frac{\sigma}{3} \pi i}, \;\; \beta_+ =\kappa  e^{\frac{1}{6}\pi i (\sigma -1)}. $
\end{itemize}
In all cases $\sigma = \pm 1$, $\kappa=\pm 1$ are  parameters.
\end{lemma}

\begin{proof}
First of all, observe that since the action of $h$ is diagonal on the Hilbert space, we can consider 
separately the diagonal and the off-diagonal parts of $D$, denoted $D_d$ and $D_o$, respectively.
For the diagonal part we have:
$$
\begin{aligned}[l]
[ \rho(h),  D_d ] & e_{\mu_1,\mu_2,\mu_3, \pm} = \\
&= e^{\frac{2\pi i}{N} (\mu_1 \pm \epsilon_1)}  
\left( \pm \beta_\pm  R \mu_1 \mp \beta_\pm R \mu_1 \right)
e_{ \mu_1, A[\mu_2 , \mu_3], \pm} = 0.
\end{aligned}
$$
whereas for the off-diagonal part:
$$
\begin{aligned}[c]
   & [ \rho(h), D_o ]  e_{\mu_1,\mu_2,\mu_3, +} = \\
= & e^{\frac{2\pi i}{N} \mu_1}  
\left( e^{-\frac{2\pi i}{N} \epsilon_1}  \beta_-  (\mu_2 + \tau^* \mu_3) - e^{\frac{2\pi i}{N} \epsilon_1} 
\beta_+ ( (A \boldsymbol{\mu})_2 + \tau^* (A \boldsymbol{\mu})_3 ) \right) e_{ \mu_1, A[\mu_2 , \mu_3], -}.
\end{aligned}
$$
To see the restrictions we need to require that the above expression vanishes for all 
$\mu_1,\mu_2,\mu_3$. Note that the action on the other part of the Hilbert space (spanned by 
$e_{\mu_1,\mu_2,\mu_3, -}$) will lead  to equivalent conditions. Next we proceed case by case.

\begin{itemize}
\item[$\mathbf{N=2}$:\ ]
The condition is:
$$ e^{-\pi i \epsilon_1}  \beta_-  ( \mu_2 + \tau^* \mu_3) - e^{\pi i \epsilon_1} 
\beta_+ ( (-\mu_2) + \tau^* (- \mu_3) )  = 0, $$
and it is satisfied if and only if $\beta_- = - \beta_+ e^{2\pi i \epsilon_1}$. Taking 
into account that $\beta_\pm$ must be real ($\beta_\pm^2=1$) and therefore equal to each other
we see that this is possible only if $\epsilon_1 = \oh$, and there are
no further restrictions on $\beta_+$.

\item[$\mathbf {N=3}$:\ ]

Here we have:
$$ e^{-\frac{2}{3} \pi i \epsilon_1}  \beta_-  ( \mu_2 + \tau^* \mu_3) - e^{\frac{2}{3} \pi i \epsilon_1} 
\beta_+ ( (- \mu_2 - \mu_3) + \tau^* (\mu_2) )  = 0, $$
Using $\beta_- = (\beta_+)^*$ this implies:
$$ 
\begin{aligned}
&   e^{\frac{4}{3}\pi i \epsilon_1}  \beta_+ (-1 + \tau^*) = (\beta_+)^*, \\ 
& - e^{\frac{4}{3}\pi i \epsilon_1}  \beta_+ = (\beta_+)^* \tau^*. 
\end{aligned}
$$
The nontrivial solution, which exists due to the restriction (\ref{taucon})
and satisfies $(\beta_+)^3 = 1$ and $\beta_+ (\beta_+)^* = 1$ is:
$$ 
\tau = e^{ -\frac{1}{3} \pi i \sigma}, \;\;\;\;
\beta_+ = - e^{ \frac{4}{3} \pi i \epsilon_1} \tau = 
e^{ \frac{2}{3} \pi i (2 \epsilon_1 + \sigma)}, $$
where $\sigma = \pm 1$.

We see that the coefficients $\beta_\pm$ are in fact fixed by the 
choice of the spin structure $\epsilon_1$. The potential additional
freedom comes from the choice of the Dirac on $\CA(\T^3_\theta)$
and can be seen in the choice of the admissible parameters $\tau$.

\item[$\mathbf{N=4}$:\ ]

In this case we obtain for any $\mu_2, \mu_3$:
$$ e^{\pi i \epsilon_1} \beta_+ \left( -(\mu_3 ) + \tau^* \mu_2 \right)    
= \beta_- (\mu_2 +\tau^* \mu_3). $$
Using $\beta_+ = (\beta_-)^*$ this implies:
$$ 
 e^{\pi i \epsilon_1}  \beta_+ =- \tau^* (\beta_+)^*, 
$$
Since $(\beta_\pm)^4 = 1$ then $(\beta_+)^2 = \pm 1$, but:
$$ (\beta_+)^2 = - \tau e^{\pi i \epsilon_1}, $$
and because $\tau = \sigma i$, $\sigma = \pm 1$, the equality 
is possible only if $\epsilon_1 = \oh$.

In the end we have:
$$ \tau = e^{ \frac{\sigma}{2} \pi i}, \;\;\;  
    \beta_+ = \pm e^{\frac{1}{4} \pi i (\sigma -1)} . 
$$

\item[$\mathbf{N=6}$:\ ]

This can only be equal if for any $\mu_2, \mu_3$:
$$ e^{\frac{2}{3}\pi i \epsilon_1} \beta_+ \left(  -\mu_3 
+ \tau^* (\mu_2+\mu_3)  \right)    = \beta_- (\mu_2 +\tau^* \mu_3). $$
Using $\beta_+ = (\beta_-)^*$ this implies:
$$ 
\begin{aligned}
&   e^{\frac{2}{3}\pi i \epsilon_1}  \beta_+ (-1 + \tau^*) = (\beta_+)^* \tau^*, \\ 
&   e^{\frac{2}{3}\pi i \epsilon_1}  \beta_+ \tau^* = (\beta_+)^*. 
\end{aligned}
$$

This can be solved:
$$ \tau =  e^{\frac{\sigma}{3} \pi i}, \;\;\;\; 
    \beta_+ =  \pm e^{\frac{1}{6} \pi i (\sigma - 2 \epsilon_1) },
    \;\; \sigma = \pm 1.
$$
However, using further $(\beta_\pm)^6=1$ we obtain:
$$ \sigma - 2\epsilon_1 \in 2 \Z,$$
which enforces $\epsilon_1=\oh$ and then:
$$ \beta_+ = \pm  e^{\frac{1}{6}\pi i (\sigma -1)}. $$
\end{itemize}
\end{proof}

\section{Real Flat Spectral Triples}

In this section we classify all real spectral triples over Bieberbach manifolds, which arise 
from flat real spectral triples over the noncommutative torus. Let us recall once again
that {\em flat} means full equivariance with respect to the action of the full isometry group of
the noncommutative torus $\CA(\T^3)_\theta$, which is $U(1)^3$. 
To simplify the notation we shall need the notion of a generalized Dirac-type
operator on the circle, $D_{\alpha,\beta}$, which is an operator with the eigenvalues:
$$ \lambda_{k} = \alpha (k+\beta), \;\;\; k \in \Z, \alpha,\beta \in \R, $$
where $\alpha \in \R$ and $-1<\beta < 1$. 
We shall denote its spectrum by ${{\mathcal S}p}^1_{\alpha,\beta}$. The $\eta$ invariant 
of this operator (see \cite{Pfa}, Lemma 5.5) is:
$$ \eta(D_{\alpha,\beta}) = \mathrm{sgn}(\beta)- 2 \beta. $$

If we have an operator with the same spectrum, however, with a multiplicity
$M>1$, then the $\eta$ invariant is $M$-multiple of the computed value. 

Using an analogous notation, we shall denote the spectrum of the Dirac
operator (for a given spin structure over $\CA(\T^3_\theta)$ by ${{\mathcal S}p}^3_{\epsilon_1,\epsilon_2,\epsilon_3}$, 
with the usual multiplicities. In case
the multiplicities are changed we shall introduce a factor in front.

\subsection{$\mathbf{N=2}$}
Let us recall that:
$$ \rho(h)e_{\mu_1,\mu_2,\mu_3,\pm}=\sigma(-1)^{\mu_1\pm\oh}e_{\mu_1,-\mu_2,-\mu_3,\pm}.$$

We have eight possibilities of real spectral triples. First, the choice of the 
spin structures over the noncommutative torus given by $\epsilon_2$ 
and $\epsilon_3$, then the choice of the sign of $\beta_+=\pm 1$.

We need to distinguish two cases. 

\subsubsection{$\epsilon_2=\oh$ or $\epsilon_3=\oh$}
Fixing $\beta_+=\sigma = \pm1$ we have the invariant subspace of the Hilbert space 
of the spectral triple over the noncommutative torus spanned by the following vectors:
\begin{equation}
 \frac{1}{\sqrt{2}} 
     \left( e_{2 k \mp\oh +j , \mu_2, \mu_3, \pm} 
            +(-1)^j \sigma e_{2 k\mp\oh +j, -\mu_2, -\mu_3, \pm} \right), 
\label{vecz21}
\end{equation}
 for $j=0,1$, with $k \in \Z$, $\mu_i \in \Z + \epsilon_i$, $i=2,3$.
The spectrum of the Dirac operator, when restricted to that spaces
consists of:
\begin{equation}
 \lambda = \pm \sqrt{ R^2 (2 k\mp\oh+ j)^2 + |\mu_2 + \tau \mu_3|^2 }. 
 \label{lamz21}
 \end{equation}
As one can see, there is no asymmetry in the spectrum, hence the $\eta$ 
invariant vanishes. In fact the spectrum of this Dirac is just the same as 
the spectrum of the Dirac on the noncommutative torus, with the 
multiplicities halved, so it is $\oh {{\mathcal S}p}^3_{\oh,\epsilon_2,\epsilon_3}$.

\subsubsection{$\epsilon_2=0$ and $\epsilon_3=0$}
Clearly, for $\mu_2 \not=0$ or $\mu_3\not=0$ the vectors (\ref{vecz21}) 
are still the invariant vectors, the spectrum of the Dirac restricted to that 
subspace is still given by (\ref{lamz21}). This part of the spectrum is, however,
not the entire spectrum of the Dirac but only its part, namely:
$$ \frac{1}{2} \left( {{\mathcal S}p}^3_{\oh,0,0} \setminus 2 {{\mathcal S}p}^1_{R, \oh} \right), $$
which means that we are not counting the spectrum for eigenvectors with
$\mu_2=\mu_3=0$. In the latter case we have the following invariant subspaces:
\begin{equation}
e_{2 k \mp \oh \sigma, 0, 0, \pm}, 
\label{vec3}
\end{equation}
and the spectrum of the Dirac operator, restricted to that subspace consists
of the following numbers:
\begin{equation}
\lambda = \pm R (2 k \mp \oh \sigma),  
 \label{lamz22}
 \end{equation}
for $k \in \Z$. It is easy to see that these spectra are:
$$ \lambda_- = R (2k + \oh), \;\;\; \sigma=-1, k \in \Z,$$
which corresponds to ${{\mathcal S}p}^1_{2R, \frac{1}{4}} $, and 
$$ \lambda_+ = R (2k - \oh), \;\;\; \sigma=1, k \in \Z,$$
which gives ${{\mathcal S}p}^1_{2R, \frac{3}{4}}$. In each case the
multiplicity of the spectrum is $2$. 

The spectra give different $\eta$ invariant:
$$ \eta(D_{\sigma}) =-\sigma \, . $$

So, in the end the spectrum is:

$$ \frac{1}{2} \left(  {{\mathcal S}p}^3_{\oh,0,0} \setminus {{\mathcal S}p}^1_{R, \oh } \right) 
\cup 2 {{\mathcal S}p}^1_{2R, \pm \oh }. $$

\subsection{$\mathbf{N=3}$}
All possible representations are:

$$\rho(h)e_{\mu_1,\mu_2,\mu_3,pm}=(-1)^{2\epsilon_1}e^{\frac 23 \pi i(\mu_1\pm \sigma)}e_{\mu_1,-\mu_2-\mu_3,\mu_2,\pm}.$$
Here we have $\epsilon_2=\epsilon_3=0$. For $\mu_2,\mu_3 \not= 0$ the invariant 
vectors are:
\begin{equation}
\begin{aligned}
\frac{1}{\sqrt{3}}  
    \left( e_{3 k + 3 \epsilon_1   \mp \sigma +j, \mu_2, \mu_3, \pm} 
  +       e^{\frac{2}{3} j \pi i}  e_{3 k+ 3 \epsilon_1   \mp \sigma +j, -\mu_2 -\mu_3, \mu_2, \pm}
  \right. \\
  \phantom{xxxxxxxxxxxxxxxx}
  +    \left.   e^{ \frac{4}{3} j \pi i}  e_{3 k+3
 \epsilon_1   \mp \sigma +j, -\mu_2-\mu_3, \pm} 
  \right)
\label{vecz32}
\end{aligned}
\end{equation}
for $j=0,1,2$.
 
The
 spectrum on this subspace is:
 \begin{equation}
 \lambda = \pm \sqrt{ R^2 (3 k+ 3 \epsilon_1 \mp \sigma  +j)^2 
 + (\mu_2)^2 + (\mu_3)^2  + \mu_2 \mu_3 },
 \label{lamz31}
 \end{equation}
 and, as a set it is (independently of $\sigma$)
$$  \frac{1}{3} \left( {{\mathcal S}p}^3_{\epsilon_1,0,0} \setminus 
     2 {{\mathcal S}p}^1_{R, \epsilon_1} \right). $$
In the case with $\mu_2=\mu_3=0$  we have the following
invariant eigenvectors:
\begin{equation}
e_{3 k+3 \epsilon_1  \mp \sigma, 0, 0, \pm}, 
\label{vecz32a}
\end{equation}
so that they are eigenvectors of $D$ to the eigenvalues:
\begin{equation}
\pm R (3 k+3 \epsilon_1 \mp \sigma),  \; k \in \Z.
\label{vecz32b}
\end{equation}
and the
 spectrum of the Dirac operator, restricted to that subspace is  the set:
$$ 2 {{\mathcal S}p}^1_{3R,  \epsilon_1 - \frac{\sigma}{3} }\; .  $$
The $\eta$ invariant is, in each of the four possible cases:
$$
\begin{aligned}
& \sigma=+1,\;\; \epsilon_1=\oh : \;\;\; & \eta =  \frac{4}{3} , \\
& \sigma=+1,\;\; \epsilon_1=0 : \;\;\; & \eta =  - \frac{2}{3}, \\
& \sigma=-1,\;\; \epsilon_1=\oh : \;\;\; & \eta =  -\frac{4}{3} , \\
& \sigma=-1,\;\; \epsilon_1=0 : \;\;\; & \eta =    \frac{2}{3}. 
\end{aligned}
$$
It is no surprise that some of the $\eta$ invariants differ by sign as, in fact,
the change $\tau \to \tau^*$ corresponds to the change $D \to -D$ on
the subspace considered and gives, in fact, the same geometry.

Therefore we have in the end two distinct spin structures for each
choice of the Dirac, each projected out of different spin structure from 
$\CA(\T^3_\theta)$, which are distinguishable by the $\eta$ invariants of the
Dirac operators. The case $\sigma=1$ is the situation discussed in 
\cite{Pfa}.
 
\subsection{$\mathbf{N=4}$} 
We begin with the recall of the representation:
$$ \rho(h)e_{\mu_1,\mu_2,\mu_3,\pm}=\kappa e^{\frac{\pi i}{2}(\mu_1\pm \frac{\sigma}{2})}e_{\mu_1,-\mu_3,\mu_2,\pm},$$
where $\sigma=\pm 1$ was defined earlier and $\kappa=\pm 1$.
 
 Here we have $\epsilon_1=\oh$ and a very similar situation to that of 
 $N=2$. We have two cases:
 
 \subsubsection{$\epsilon_2=\epsilon_3=\oh$}

The invariant subspace of the Hilbert space of the spectral triple over 
the noncommutative torus is spanned by the following vectors:
\begin{equation}
 \frac{1}{2}  \sum_{p=0}^3
     \left( \kappa^j e^{\frac{j}{2} \pi i} \rho^p(h)(e_{4 k \mp \frac{\sigma}{2}+j, \mu_2, \mu_3, \pm}) \right), 
\label{vecz41}
\end{equation}
for $k \in \Z$ and $j=0,1,2,3$, where 
$\sigma = \pm 1, \kappa =\pm 1$.

The spectrum is symmetric, with the eigenvalues:
\begin{equation}
 \lambda = \pm \sqrt{ R^2 (4 k \mp \oh \sigma + j )^2 + |\mu_2|^2 + |\mu_3|^2 }, \;\; 
 k \in \Z, j=0,1,2,3.
 \label{lamz41}
 \end{equation}

and it is clear to see that it is, in fact, the original spectrum of the Dirac on the 
$\mathcal{A}(\T^3_\theta)$, with $\frac{1}{4}$ of its multiplicities, 
$\frac{1}{4} {{\mathcal S}p}^3_{\oh,\oh,\oh}$.
  
\subsubsection{$\epsilon_2 = \epsilon_3=0$}

Again, similarly as in the $n=2$ case if $\mu_2 \not=0$ (note here 
this enforces $\mu_3\not=0$) we have the same part of the  spectrum, 
on the subspace spanned by the same vectors (\ref{vecz41}), which are 
still the invariant vectors for $\mu_2 \not=0$.

Again, this is only a part of the spectrum of the Dirac on $\CA(\T^3_\theta)$, namely:
$$ \frac{1}{4} \left( {{\mathcal S}p}^3_{\oh,0,0} \setminus 2 {{\mathcal S}p}^1_{R, \oh } \right), $$
which means that we are not counting the spectrum for eigenvectors with
$\mu_2=\mu_3=0$ and the remaining part of the spectrum comes with the multiplicity
$1$ instead of $4$.

In the latter case, direct computation shows that the following vectors
space the invariant subspaces:
\begin{equation}
e_{4 k \mp \oh \sigma-1+\kappa , 0, 0, \pm}, \;\;\; 
\label{vecz4}
\end{equation}
and the spectrum of the Dirac operator, restricted to that subspace contains
all the following eigenvalues:
\begin{equation}
\lambda =R (4 k- \oh \sigma
 + 1-\kappa),  \;\;\; k \in \Z.
 \label{lamz42}
 \end{equation}
 
We have, in all possible cases $\sigma = \pm 1, \kappa =\pm 1$ the spectra
give the following $\eta$ invariants:

$$
\begin{aligned}
& \sigma=+1,\;\; \kappa=+1 : \;\;\; & \eta =  - \frac{3}{2} , \\
& \sigma=+1,\;\; \kappa=-1:  \;\;\; & \eta =      \frac{1}{2}, \\
& \sigma=-1,\;\; \kappa=+1:  \;\;\; & \eta =   \frac{3}{2} , \\
& \sigma=-1,\;\; \kappa=-1:  \;\;\; & \eta =   - \frac{1}{2}. 
\end{aligned}
$$
 
Again, the case $\sigma=+1$ and $\sigma=-1$ are related by the map
$D \to - D$, the case presented in \cite{Pfa} correspond to $\sigma=-1$.

\subsubsection{$\mathbf{N=6}$}

We recall the representation:

$$ \rho(h)e_{\mu_1,\mu_2,\mu_3,\pm}= \kappa e^{\frac{\pi i}{3}(\mu_1\pm \frac{\sigma}{2})}e_{\mu_1,-\mu_3,\mu_2+\mu_3,\pm},$$
where $\sigma=\pm 1$, $\kappa=\pm 1$.

Here, one repeats most of the arguments from the $N=3$ case.First of all,
there exists a part of the invariant subspace of the Hilbert space where the
spectrum of $D$ is symmetric and is:
$$ \frac{1}{6} \left( {{\mathcal S}p}^3_{\epsilon_1,0,0} \setminus 
2 {{\mathcal S}p}^1_{R,  \epsilon_1} \right). $$

We do not write explicit expression for the vectors spanning this subspace, the
formula is analogous to the ones derived earlier for $N=3$.

Similarly, there exists additional invariant vectors:

$$ e_{6k \mp \oh \sigma -\frac 32 +\frac 32 \kappa , 0, 0 \pm}, \;\;\; k \in \Z, $$
where $\sigma =\pm 1$and $ \kappa = \pm 1$.

The eigenvalues are:

\begin{equation}
\lambda = \pm R (6 k -\frac 12 \sigma + 3\frac{1-\kappa}{2}),  \;\;\; k \in \Z
 \label{lamz42}
 \end{equation}
 and this gives the spectrum 
 $2 {{\mathcal S}p}^1_{6R, \frac{1-\kappa}{4}-\frac{\sigma}{12}}$.

We have, in all possible cases $\sigma = \pm 1, \kappa =\pm 1$ the spectra
give the following $\eta$ invariants:

$$
\begin{aligned}
& \sigma=+1,\;\; \kappa=+1 : \;\;\; & \eta =  - \frac{5}{3} , \\
& \sigma=+1,\;\; \kappa=-1: \;\;\; & \eta =    \frac{1}{3}, \\
& \sigma=-1,\;\; \kappa=+1: \;\;\; & \eta =     \frac{5}{3} , \\
& \sigma=-1,\;\; \kappa=-1: \;\;\; & \eta =   - \frac{1}{3}. 
\end{aligned}
$$
Again, we see that the case presented in \cite{Pfa} is the one $\sigma=-1$.
 
To summarize we can present the result of our classification, the table below gives 
the number of distinct flat real spectral triples over noncommutative Bieberbachs: 

\begin{table}[here]
\centering
\begin{tabular}{|c|c|c|}
\hline
Bieberbach    & Parametrisation & Number of \\
space &&spin structures \\ \hline \hline
${\mathrm B}2_\theta$   & $\epsilon_2,\epsilon_3= 0, \frac{1}{2} $; $\sigma=\pm 1$    &$ 8 $   \\ \hline
${\mathrm B}3_\theta$   &  $\epsilon_1= 0, \frac{1}{2}$  & $2$\\ \hline
${\mathrm B}4_\theta$   &  $\epsilon_2=\epsilon_3 = 0, \frac{1}{2}$; $\kappa=\pm 1$& $4$ \\ \hline
${\mathrm B}6_\theta$   &  $\kappa=\pm 1$ & $2$\\ \hline
\end{tabular}
\caption{Noncommutative spin structures over ${\mathrm B}N_\theta$}
\label{ssncbieb}
\end{table}

\subsection{Irreducibility}

The notion of irreducibility of real spectral triples is a delicate one and one can distinguish
at least three classical notions of irreducibility:
\begin{itemize}
\item {\em Full irreducibility}: there is no proper subspace of the Hilbert space, invariant
under the algebra, $D$ and $J$.
\item {\em $J$ irreducibility}: there is no proper subspace of the Hilbert space, invariant
under the algebra and $J$.
\item {\em $D$ irreducibility}: there is no proper subspace of the Hilbert space, invariant
under the algebra and $D$.
\end{itemize}

Observe that the canonical equivariant spectral triple for the noncommutative torus 
$\mathcal{A}(\T^3_\theta)$ is irreducible in either way.

It is easy to see that the constructed real spectral triples over Bieberbachs are irreducible 
in each of the above ways. Indeed, assuming the contrary (in either of the above cases)
and identifying the Hilbert space of the constructed triples with a subspace of the triple 
for the noncommutative torus, one would be able to construct a subspace of that Hilbert 
space, which would be invariant by the full algebra of the torus and $D$ and/or $J$, 
respectively. We briefly sketch here the method: if $\CH_r \subset \CH_B$  
is such a subspace reducing the triple over Bieberbachs identified as a subspace 
of the Hilbert space for the torus triple, one just needs to take $\oplus_{k=0}^{N-1} U^k \CH_r$. 
This will be a proper subspace of the Hilbert space of the spectral triple over noncommutative 
torus, as the spaces  $U^k \CH_B$, where  are disjoint for each $0 \leq k < N$.

Let us remark, however, that the triples constructed earlier are all the only ones {\em flat} and irreducible (in all three ways). The assumption of flatness was necessary, as only in such case 
there exist a classification of real spectral triples over a noncommutative torus. Although one
can easily construct triples, which are reductions of the flat triple over noncommutative torus,
(basically by creating a direct sum of restrictions to $U^k \CH_B$ within the Hilbert space 
of the triple over $\mathcal{A}(\T^3_\theta)$) they will necessarily be the reducible ones in the sense of
$D$-reducibility. For this reason they are not described in the paper.

\section{Conclusions}

The examples we have provided are interesting for geometry and noncommutative geometry 
for several reasons. First of all, noncommutative Bieberbach manifolds, though only 
mildly noncommutative, are not examples of $\theta$-deformations,  since classical Bieberbach 
manifolds do not admit a free action of the torus. Therefore we have a genuine new class of 
noncommutative examples. Furthermore, our computations of real spectral triples attempts to 
answer the natural question about the existence of different spin structures and their explicit 
construction. Our paper, building on the results from \cite{PaSi} and \cite{Vens} shows that for this
interesting class of noncommutative geometries, with the a additional assumption (which states
that we are looking for flat, irreducible geometries) we can provide a complete classification.

The classical limit of presented noncommutative geometries might have a place in the studies
of models of flat compact space geometries in the context of cosmology \cite{OlSi1}. The construction
of different inequivalent spectral triples might suggest that there might be some differences between
particle models built using them, in particular, twisting the obtained geometries with 
the projective modules, which represent nontrivial torsion in the K-theory of these objects 
\cite{OlSi2}. 


\vspace{1cm}
\end{document}